\documentclass[11pt]{amsart}
\usepackage{graphicx,amsmath,amssymb,latexsym,bbm,mathtools,wasysym}
\usepackage{enumerate}
\usepackage{fullpage}
\usepackage{subfigure}
\usepackage{todonotes}
\usepackage[normalem]{ulem}

\makeatletter
\@namedef{subjclassname@2020}{%
  \textup{2020} Mathematics Subject Classification}
\makeatother

\newtheorem{theorem}{Theorem}[section]

\newtheorem{lemma}[theorem]{Lemma}
\newtheorem{proposition}[theorem]{Proposition}
\theoremstyle{definition}
\newtheorem{definition}[theorem]{Definition}
\theoremstyle{remark}
\newtheorem{remark}[theorem]{Remark}
\theoremstyle{remark}
\newtheorem{example}{Example}
\numberwithin{equation}{section}

\newcommand{\set}[1]{\left\{#1\right\}}
\newcommand{\R}{\mathbb R}

\newcommand{\N}{\mathbb N}

\newcommand{\PP}{\mathbb P}
\newcommand{\UU}{\mathbbm 1}
\newcommand{\BB}{{\mathcal B}}

\newcommand{\FF}{{\mathcal F}}

\newcommand{\PPP}{{\mathcal P}}

\newcommand{\eps}{\varepsilon}
\newcommand{\mixset}{rapid $\psi$-mixing}

\begin{document}

\title[]{Almost sure limit theorems with applications to non-regular continued fraction algorithms}
\author{Claudio Bonanno}
\address{Dipartimento di Matematica, Universit\`a di Pisa, Largo Bruno Pontecorvo 5, 56127 Pisa, Italy}
\email{claudio.bonanno@unipi.it}

\author{Tanja I. Schindler}
\address{Faculty of Mathematics and Computer Science, Jagiellonian University in
Krakow, Poland and\\ Department of Mathematics and Statistics, University of Exeter, UK and\\
Faculty of Mathematics, University of Vienna, Oskar-Morgenstern-Platz 1, 1090 Vienna, Austria}
\email{tanja.schindler@uj.edu.pl, t.schindler@exeter.ac.at, tanja.schindler@univie.ac.at}

\begin{abstract}
We consider a conservative ergodic measure-preserving transformation $T$ of the measure space $(X,\BB,\mu)$ with $\mu$ a $\sigma$-finite measure and $\mu(X)=\infty$. Given an observable $g:X\to \R$, it is well known from results by Aaronson, see \cite{aa-book}, that in general the asymptotic behaviour of the Birkhoff sums $S_Ng(x):= \sum_{j=1}^N\, (g\circ T^{j-1})(x)$ strongly depends on the point $x\in X$, and that there exists no sequence $(d_N)$ for which $S_Ng(x)/d_N \to 1$ for $\mu$-almost every $x\in X$. In this paper we consider the case $g\not\in L^1(X,\mu)$ and continue the investigation initiated in \cite{no1}. We show that for transformations $T$ with strong mixing assumptions for the induced map on a finite measure set, the almost sure asymptotic behaviour of $S_Ng(x)$ for an unbounded observable $g$ may be obtained using two methods, addition to $S_Ng$ of a number of summands depending on $x$ and trimming. The obtained sums are then asymptotic to a scalar multiple of $N$. The results are applied to a couple of non-regular continued fraction algorithms, the backward (or R\'enyi type) continued fraction and the even-integer continued fraction algorithms, to obtain the almost sure asymptotic behaviour of the sums of the digits of the algorithms. 
\end{abstract}

\subjclass[2020]{37A40, 37A25, 60F15, 11K50}
\keywords{Infinite ergodic theory; almost sure limits for Birkhoff sums; trimmed sums; non-regular continued fraction algorithms}
\thanks{CB is partially supported by 
the PRIN Grant PRIN 2022NTKXCX ``Stochastic properties of dynamical systems'' funded by the Ministry of University and Scientific Research, Italy. TS is partially supported by the Austrian Science Fund FWF: P 33943-N and by the
MSCA Project ErgodicHyperbolic - p.n.\ 101151185.\\ CB acknowledges the MIUR Excellence Department Project awarded to the Department of Mathematics, University of Pisa, CUP I57G22000700001.\\
This research is part of the authors' activity within the UMI Group ``DinAmicI'' \texttt{www.dinamici.org} and of CB's activity within the 
Gruppo Nazionale di Fisica Matematica, INdAM, Italy. TS acknowledges the support of the University of Pisa through the ``visiting fellows'' program, as this work was partially done during her visit at the Dipartimento di Matematica of the University of Pisa.}

\maketitle

\section{Introduction} \label{sec:intro}

Let $T:X\to X$ be a conservative ergodic measure-preserving transformation of the measure space $(X,\BB,\mu)$ with $\mu$ a $\sigma$-finite measure. Let $g:X\to\mathbb{R}$ be a measurable observable on $X$ and define the \emph{Birkhoff sums}
\[
S_Ng(x):=\sum_{n=1}^N\, (g\circ T^{n-1})(x).
\]
In case that $\mu$ is a probability measure and $g\in L^1(X,\mu)$ we get by Birkhoff's Ergodic Theorem for $\mu$-almost every (a.e.) $x\in X$ that 
\[
\lim_{N\to\infty} \frac{S_Ng(x)}{N}=\int g\, \mathrm{d}\mu\,.
\]
However, in the situation that $\mu$ is a probability measure and $g\not\in L^1(X,\mu)$ or in the situation that $\mu(X)=\infty$ and $g\in L^1(X,\mu)$, by Aaronson's theorem, \cite{aa-book}, we cannot obtain an analogous strong law of large numbers, i.e. we have for any norming sequence $(d_N)$ and $\mu$-a.e. $x\in X$ that a general non-negative observable satisfies
\begin{align*}
 \limsup_{N\to\infty}\frac{S_Ng(x)}{d_N}=\infty \quad \text{ or }\quad
 \liminf_{N\to\infty}\frac{S_Ng(x)}{d_N}=0.
\end{align*}

In the finite measure case (also in the case that we only consider i.i.d. random variables) one method to still obtain some information on the almost sure limit behaviour is to use \emph{trimming}: 

Given a sequence $(Y_n)$ of random variables on a probability space $(\Omega,\PP)$ and a point $\omega\in \Omega$, for each $N\in \N$ we choose a permutation $\pi$ of $\{1,2,\dots,N\}$ such that $Y_{\pi(1)}(\omega) \ge Y_{\pi(2)}(\omega) \ge \dots \ge Y_{\pi(N)}(\omega)$. For a given $r\in \N \cup \{0\}$, the \emph{lightly trimmed sum} of $(Y_n)$ is defined by
\begin{equation}\label{trimmed-birk-sum}
S_N^r(\omega) := \sum_{n=r+1}^N\, Y_{\pi(n)}(\omega) 
\end{equation}
that is the sum of the first $N$ random variables trimmed by the largest $r$ entries. In case $Y_n=g\circ T^{n-1}$, for all $n$, we also write $S_N^rg(x)$ for the trimmed Birkhoff sum at $x$. In some situations it is necessary to trim a number of entries increasing with $N$. The \emph{intermediately trimmed sum} of $(Y_n)$ is defined by
\begin{equation}\label{inter-trimmed-birk-sum}
S_N^{ r_N}(\omega) := \sum_{n=r_{N}+1}^N\, Y_{\pi(n)}(\omega) 
\end{equation}
where $(r_n)$ is a sequence satisfying $r_n \to \infty$ and $r_n/n\to 0$ as $n\to \infty$.
In the ergodic context the first strong laws under trimming have been proven for the regular continued fraction entries, see \cite{diamond_estimates_1986}, a result later also generalized for $\alpha$-continued fractions with $\alpha\geq 1/2$, see \cite{nak-nat-trim}. 
However, there are also (strong) laws of large numbers in the ergodic context dealing with a larger class of functions, see \cite{aar-nakada, kesseboehmer_strong_2019, kesseboehmer_intermediately_2019, schindler_observables_2018, kesseboehmer_mean_2019} and for related results for i.i.d.\ random variables see \cite{kestenmaller, haeusler_laws_1987, haeusler_nonstandard_1993, kesseboehmer_strongiid_2019}. 

On the contrary, for the second situation in Aaronson's theorem ($\mu$ being an infinite $\sigma$-finite measure and $g\in L^1(X,\mu)$), other methods of truncating the sum are necessary.
As shown in \cite{aar-kos-wei} considering the symmetric sum can not give almost sure convergence. However, under nice enough mixing conditions and for sufficiently regular observables, the authors of this paper have proven a strong law adding a number of summands, i.e. the existence of a sequence $(d_N)$ and a function $m:X\times \N\to\N$ such that for $\mu$-a.e. $x\in X$ we have
\begin{equation}\label{eq:no1 conv}
 \lim_{N\to\infty}\frac{S_{N+m(N,x)}g(x)}{d_N}=1\,,
\end{equation}
see \cite[Theorem 2.3]{no1}. We note that $m$ depends on $N$ and $x$ and can vary a lot for different values of $x$, but $(d_N)$ is independent of $x$. The precise definition of $m$ is recalled below in \eqref{def-m}.

Let us lastly have a look at the last, not yet discussed case, namely that $\mu(X)=\infty$ and $g\not\in L^1(X,\mu)$. In this setting there are trivial cases for a strong law to hold, e.g. one can set $g$ to be a constant $C$, then we have for all $x\in X$ that $\lim_{N\to\infty} S_Ng(x)/N=C$. In general this behaviour is not true though, see \cite{no1, lenci} for conditions in particular settings. The conditions in \cite{no1, lenci} all require implicitly $g$ to be integrable on finite-measure subsets of $X$. However, there are a number of interesting examples, for instance the observable taking as values the first digit of some generalised continued fraction expansions, which do not fall into this category, i.e. their underlying invariant measure is infinite but the interesting observable is unbounded and non-integrable on a finite-measure set. 

In this paper we investigate into a class of systems which fall into the lastly described setting and prove an almost sure limit theorem under the use of both methods - trimming and adding summands as in \eqref{eq:no1 conv}. Furthermore, we study a couple of concrete examples which fall exactly into this setting, namely the backward or R\'enyi type continued fractions and the even-integer continued fractions which were first introduced in \cite{renyi} and \cite{schw-even, schw-even-2} respectively. The latter one can also be seen as a special case of $\alpha$-continued fractions for $\alpha=0$. It might also worth mentioning that, though the backward continued fraction transformation is nothing else as a reflection at the vertical line at $x=1/2$ for the Gauss map, the ergodic properties of their digits are quite different, see e.g.\ \cite{ik-book, aar-nakada, takahasi}. 

The structure of the paper is as follows. We first introduce the general setting from \cite{no1}. Then in Section \ref{sec:results} we state our main results giving an application to the generalised continued fractions expansions in Subsection \ref{subsec: applications}. In Section \ref{sec: proof gen thm} we prove our main results and in Section \ref{sec: proofs cf} we give the proofs for our examples.

\section{The setting} \label{sec:setting}

Let $T:X\to X$ be a conservative ergodic measure-preserving transformation of the measure space $(X,\BB,\mu)$ with $\mu$ a $\sigma$-finite measure with $\mu(X)=\infty$, and let us fix a set $E\in \BB$ with $\mu(E)=1$. For $E$, we define the \emph{first return time} $\varphi_{_E}$  
\[
\varphi_{_E} : E \to \N,\qquad \varphi_{_E}(x):= \inf \set{k\ge 1\, :\, T^k(x)\in E}.
\]
The function $\varphi_{_E}$ is finite $\mu$-a.e. and induces a measurable partition of $E$ by using its level sets. Let 
\begin{equation} \label{level-sets}
A_n := \set{x\in E\, : \varphi_{_E}(x)=n},
\end{equation}
then $E = \bigsqcup_{k\ge 1}\, A_k$ up to zero measure sets. In the following we also use the super-level sets
\[
A_{>n} := \set{x\in E\, :\, \varphi_{_E}(x) > n} = \bigsqcup_{k>n}\, A_k,
\]
and the sets $A_{\ge n} = A_n \cup A_{>n}$. Applying Kac's Theorem one has
\[
\sum_{k\ge 1}\, k\, \mu(A_k) = \sum_{n\ge 0}\, \mu(A_{>n}) = \mu(X) = \infty,
\]
and the order of infinity of the previous series is an important indicator of the map $T$ which is independent on the choice of the set $E$. 
Furthermore, we define the \emph{wandering rate} 
\[
w_n(E):=\sum_{k=0}^{n-1} \mu(A_{>k}).
\]
We consider this indicator by looking at the sequence
\begin{equation} \label{alfa}
\alpha(n):= \frac{n}{w_n(E)}.
\end{equation}
In connection with the wandering rate we will also need the notion of \emph{slow variation}. A function $f:\mathbb{N}\to \mathbb{R}$ is called \emph{slowly varying} if for all $d>0$ we have $\lim_{N\to\infty}f(\lfloor dN\rfloor)/f(N)=1$, for details about this notion see \cite{reg-var-book}.

By using the first return time function, one defines the \emph{induced map} $T_{_E}:E \to E$,  an ergodic measure-preserving transformation of the probability space $(E,\BB|_E,\mu)$ given by $T_{_E}(x) := T^{\varphi_{_E}(x)}(x)$. By considering $T_{_E}$ we look at the orbits of $T$ by studying two properties: the visits to $E$ which are subsequently obtained by applying $T_{_E}$ and the excursions out of $E$. This method is particularly useful when studying the Birkhoff sums $S_N g$ of observables $g: X\to \R$.

We also recall from \cite{no1} the precise definition of  the \emph{longest excursion out of $E$ beginning in the first $N$-steps}, the quantity $m$ used in \eqref{eq:no1 conv} and defined $\mu$-a.e. in $X$ as  
\begin{equation}\label{def-m}
m(N,E,x) := 1+\max \set{ k\ge 1\, :\, \exists\, \ell \in \{1,\dots,N+1\} \text{ s.t. } T^{\ell+j}(x) \not\in E,\, \forall\, j=0,\dots,k-1}.
\end{equation}

To conclude this section, we comment on the properties needed to apply the trimming methods described in the introduction. In the literature, these methods have been applied to systems with ``good'' mixing properties. To proceed we recall the useful notions of $\psi$-mixing in the general case of random variables. We refer to \cite{bradley} for more definitions.

\begin{definition}\label{def-psi-mixing}
Let $(Y_n)$ be a sequence of random variables on a probability space $(\Omega,\PP)$, and let $\FF_h^k$, for $0\le h<k\le \infty$, be the $\sigma$-field generated by $(Y_n)_{h\le n\le k}$. The sequence $(Y_n)$ is \emph{$\psi$-mixing} if
\[
\psi(n) := \sup \left\{ \Big| \frac{\PP(B\cap C)}{\PP(B)\PP(C)} -1 \Big|\, :\, B\in \FF_0^j, \, C\in \FF_{j+n}^\infty,\, \PP(B)>0,\, \PP(C)>0,\, j\in \N \right\}
\]
satisfies $\psi(n)\to 0$ as $n\to \infty$.
\end{definition}

In our approach we need that the $\psi$-mixing condition is satisfied by the sequence of random variables $(f\circ T_{_E}^{n-1})_{n\in \N}$ for $f:E\to \R$ to be defined later, and that the sequence $(\psi(n))$ decays fast enough to apply Lemma \ref{lemma-aar-nakada}.

\begin{definition}\label{def-set-induce}
Let $E\in \BB$ be a finite measure set with $\mu(E)=1$ and such that $T_{_E}(A_n)=E$ a.s.\ for all $n\in \N$. We say that $E$ \emph{induces \mixset} if for any sequence $(\phi_n)_{n\in\mathbb{N}}$ of functions from $E$ to $\R$ such that each $\phi_n$ is piecewise constant on the partition $\mathcal{P}_{_E}$ of $E$ induced by $T_{_E}$, i.e. the finest partition such that for each element $P\in \mathcal{P}_{_E}$ we have $T_{_E}(P)=E$ a.s., the sequence $(\phi_n\circ T_{_E}^{n-1})_{n\ge 1}$ is $\psi$-mixing with coefficient $\psi(n)$ fulfilling $\sum_{n\ge 1} \psi(n)/n <\infty$.
\end{definition}

In the paper we use the following notations:
\begin{itemize}
\item For sequences $(a_{N,M})$, $(b_N)$, and $(c_M)$, we write $a_{N,M}\asymp b_N c_M$ if there exist constants $C>0$ and $N_0$ such that $C^{-1} a_{N,M}\leq b_N c_M \leq C a_{N,M}$, for all $M,N\geq N_0$ and analogously, we write $b_N\asymp c_N$ if there exist constants $C>0$ and $N_0$ such that $C^{-1} b_{N}\leq c_N \leq C b_{N}$, for all $M,N\geq N_0$. 
Moreover, we write $\sum_{N=1}^{\infty} b_N\asymp \sum_{N=1}^{\infty} c_N$ if $b_N\asymp c_N$ implying that each of the two sums is either finite, $\infty$ or $-\infty$.  
\item For sequences $(a_N)$, $(b_N)$, we write $a_N \sim b_N$ if $\lim_{N\to \infty} a_N/b_N =1$.
\item For a real function $a(y)$ defined on a neighbourhood of $+\infty$ and tending to $+\infty$ as $y\to +\infty$, the asymptotic inverse is a function $b(y)$ such that $a(b(y)) \sim b(a(y)) \sim y$ as $y\to +\infty$.
\end{itemize}

\section{The main results} \label{sec:results}

With the definitions from the previous section we are able to state our first main result. 

\begin{theorem}\label{thm:main1}
Let $T:X\to X$ be a conservative, ergodic, measure-preserving transformation of the measure space $(X,\BB,\mu)$ with $\mu$ a $\sigma$-finite measure with $\mu(X)=\infty$, and let $E\in \BB$ with $\mu(E)=1$ be a set which induces \mixset. Let $g:X \to \R_{\ge 0}$ be a measurable observable. We assume that:
\begin{itemize}
\item[(i)] The level sets fulfill 
\begin{align}
 \sum_{n\geq 1}\frac{n (\mu(A_{>n}))^2}{w_n(E)^2}<\infty\label{eq: cond level sets}
\end{align}
and $w_n(E)$ is slowly varying.
\item[(ii)] There exists a constant $\kappa>0$ such that $\mu(g>n) \sim \kappa \mu(A_{>n})$.
\item[(iii)] The function $g$ is locally constant on the partition $\PPP_{_E}$ of $E$, and $g\not\in L^1(E,\mu)$.
\item[(iv)] There exists $c\in \R$ such that $g \equiv c$ on $X\backslash E$.
\end{itemize}
Then, for $\mu$-a.e. $x\in X$ we have
\begin{equation}
\lim_{N\to \infty}\, \frac{S_{N+m(N,E,x)}g(x)- \max_{1\leq k\leq N+m(N,E,x)} (g\circ T^{k-1})(x) - c\, m(N,E,x)}{N}= c+\kappa,\label{eq:main1}
\end{equation}
where the subtraction of the last two terms in the numerator is necessary since for $\mu$-a.e. $x\in X$ we have that 
\begin{align} \label{eq-cfr}
 \limsup_{N\to \infty}\,\frac{\max_{1\leq k\leq N+m(N,E,x)} (g\circ T^{k-1})(x)}{N}
 =\limsup_{N\to \infty}\,\frac{m(N,E,x)}{N}=\infty. 
\end{align}
If, additionally to conditions (i)-(iv), we have that
\begin{align}
  \mu\left(\left\{g >N\right\} \cap \left\{\varphi_{_E} >M\right\}  \right)
  \asymp \mu\left(g >N\right) \mu \left(\varphi_{_E} >M \right)\label{eq:cond1}
\end{align}
 for $M,N\in\mathbb{N}$, then we obtain the slightly stronger result that for $\mu$-a.e. $x\in X$
 \begin{align}
 \lim_{N\to \infty}\, \frac{S_{N+m(N,E,x)}g(x)- \max\left\{\max_{1\leq k\leq N+m(N,E,x)} (g\circ T^{k-1})(x), c\, m(N,E,x)\right\}}{N}=c+\kappa.\label{eq:main2}
\end{align}
\end{theorem}

The interesting point here is that, as specified in assumption (ii), the order of infinity of the measure is the same as the order of infinity of the observable. We will see from the results in Section \ref{sec:other-res} what happens if we slightly change this. 

\begin{remark} \label{new-rem}
 Condition (iv) could be weakened to $g$ be bounded on $X\setminus E$ and satisfying the conditions in \cite[Theorem 2.7]{no1}. In that result, we have stated assumptions on the system and on an observable $g_c$ bounded on $E$ under which 
 \begin{align}\label{eq: gc asymp}
  S_N g_c(x) \sim cN.
 \end{align}
Let now $g:X \to \R_{\ge 0}$ be a measurable observable satisfying assumptions (ii) and (iii) of Theorem \ref{thm:main1} and which is bounded on $X\setminus E$. Then, setting $g_c\equiv c$ on $E$ and $g_c=g$ on $X\setminus E$, the function $f:= g-g_c$ satisfies (ii), (iii), and (iv) of Theorem \ref{thm:main1}. Moreover, if $g_c$ fulfills \eqref{eq: gc asymp} it is enough to study $S_Nf$ to obtain results on $S_Ng$. However, in order to not introduce more notation in the statement of the theorem we restrict ourselves to the easier condition (iv). This argument is fully explained in the proof, see \eqref{g-reduced}-\eqref{splitting}.
\end{remark}

\subsection{Statements for a flattened or increased observables}\label{sec:other-res} 

In this section we consider the effects of modifying assumption (ii) of Theorem \ref{thm:main1} on the result. 

Namely, under the otherwise same assumptions of Theorem \ref{thm:main1} on the system $(X,\mu,T)$ and the set $E$, we now look at situations where either $\mu(g>n)=o(\mu(A_{>n}))$ (Theorem \ref{thm:h-slow}) or $\mu(A_{>n})=o(\mu(g>n))$ (Theorem \ref{thm:h-fast}). The asymptotic behaviour of the Birkhoff sums in \eqref{eq:main1} is different and simpler if $\mu(g>n)$ is slower or faster than $\mu(A_{>n})$. In the former case, we obtain that it is enough to consider the lightly trimmed sum $S^1_Ng$ and under certain conditions it is even possible to obtain a strong law of large numbers where we neither have to add additional summands nor to trim the sum.

\begin{theorem}\label{thm:h-slow}
Let the space $(X,\BB,\mu)$, the transformation $T:X\to X$, the set $E\in \BB$ and $g:X \to \R_{\ge 0}$ a measurable observable be as in Theorem \ref{thm:main1}, 
with the sets $A_{>n}$ and the wandering rate $w_n(E)$ satisfying assumptions (i), (iii), and (iv) of Theorem \ref{thm:main1}.\\ If furthermore we assume
\begin{itemize}
\item[(iia)] $\mu(g>n) = o(\mu(A_{>n}))$. 
\end{itemize}
Then, for $\mu$-a.e. $x\in X$ we have
\begin{equation} \label{slow-first-r}
\lim_{N\to \infty}\, \frac{S_N g(x) - \max_{1\le k\le N}\, (g \circ T^{k-1})(x)}{N} = c.
\end{equation}
A stronger result is obtained under a stronger assumption. Let $\beta$ be the asymptotic inverse of $\alpha$ given in \eqref{alfa}, and additionally to (i), (iii), and (iv) of Theorem \ref{thm:main1}, let us assume
\begin{itemize}
\item[(ii\~{a})] $\mu(g>n) = o(\mu(A_{>n}))$ and $\sum_{n=1}^{\infty}\mu(g>\epsilon \beta(n))<\infty$ for all $\epsilon>0$. 
\end{itemize}
Then, for $\mu$-a.e. $x\in X$ we have
\begin{equation} \label{slow-second-r}
\lim_{N\to \infty}\, \frac{S_N g(x)}{N} = c.
\end{equation}
\end{theorem}
 
\begin{remark}
 It is easy to construct examples for which $\int_E g\,\mathrm{d}\mu=\infty$ and which fulfill all the strongest assumptions of Theorem \ref{thm:h-slow}. Let e.g.\ 
 $\mu(A_{>n})\sim 1/n$ as in the special continued fraction transformations described in the next subsection
 and let $g$ be such that $\mu(g>n)=1/ (n(\log\log n)^2)$. Obviously $g$ is not integrable. However, we have that $\alpha(n)\sim n/\log(n)$ and thus $\beta(n)\sim n\log(n)$. Hence,
 $\mu(g>\eps \beta(n))\leq 2/(\eps n\log (n)(\log\log (n))^2)$, for $n$ sufficiently large. Thus, 
 $\sum_{n=1}^{\infty} \mu(g>\eps \beta(n))<\infty$ and (ii\~{a}) is fulfilled. 
\end{remark}

We now consider the case of $\mu(g>n)$ faster than $\mu(A_{>n})$. As expected we need to trim the Birkhoff sums of $g$ but we also need to consider the sums up to time $N+m(N,E,x)$. The result of the limit however does not depend on the value of $g$ on $X\setminus E$. 

\begin{theorem}\label{thm:h-fast}
Let the space $(X,\BB,\mu)$, the transformation $T:X\to X$, the set $E\in \BB$
and $g:X \to \R_{\ge 0}$ a measurable observable be as in Theorem \ref{thm:main1}, 
with the sets $A_{>n}$ and the wandering rate $w_n(E)$ satisfying assumptions (i), (iii) and (iv) of Theorem \ref{thm:main1}. Furthermore, let us assume
\begin{itemize}
\item[(iib)] $\mu(g>n)/\mu(A_{>n}) \to \infty$. 
\end{itemize}
Furthermore, for $F(y)=1-\mu(g>y)$ assume that for some $y_0>1$ the quantity
\[
 W\coloneqq \inf \left\{r\in\mathbb{N}\colon \int_{y_0}^{\infty}\left(\frac{y\left(1-F(y)\right)}{\int_{y_0}^y\left(1-F(t)\right)\mathrm{d}t}\right)^{r+1} \frac{1}{y}\, \mathrm{d}y<\infty\right\}
\]
is finite. Then, we have for $\mu$-a.e.\ $x\in X$
\[
\lim_{N\to \infty}\, \frac{S^{W}_{N+m(N,E,x)} g(x) - c\, m(N,E,x)}{b(\alpha(N))} = 1,
\]
where $b(n)$ is the asymptotic inverse function of $a(y):=y/\int_{y_0}^y  \mu(g>t)\,\mathrm{d}t$, and $\alpha(n)$ is defined in \eqref{alfa}. 
\end{theorem}

\subsection{Applications for non-regular continued fractions}\label{subsec: applications}

We now show an application to the sums of the coefficients of two continued fraction expansions, the backward and the even-integer continued fraction expansions. Analogous results can be proved with the same techniques for other expansions such as the odd-odd one (see \cite{lee}). 

\begin{example}[Backward (or R\'enyi type) Continued Fractions] \label{ex:bcf}
We define the Backward Continued Fraction transformation $T_{BCF}:[0,1]\to[0,1]$ by 
\[
T_{BCF}(x)=\left\{\frac{1}{1-x}\right\},
\]
where $\{x\} = x$ (mod 1), see Figure \ref{cf-fig}-(a).

 \begin{figure}[ht!]
 \begin{center}
 \subfigure[]
     {\includegraphics[width=5cm]{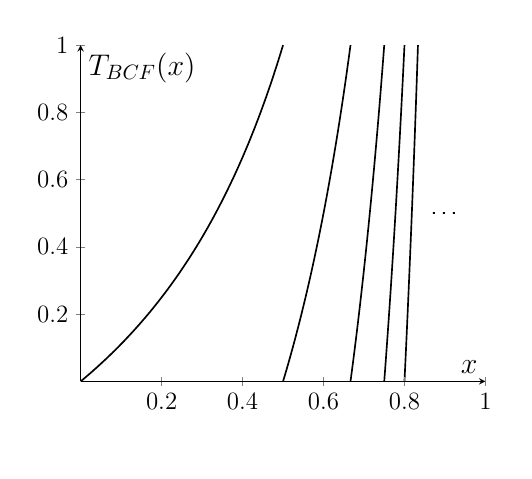}}
     \hspace{0.7cm}
     \subfigure[]
     {\includegraphics[width=6cm]{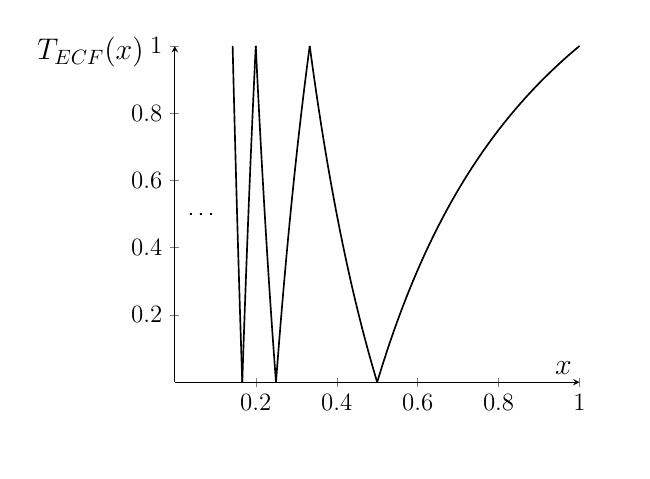}}
 \caption{(a) The Backward Continued Fraction transformation $T_{BCF}$.  (b) The Even-Integer Continued Fraction transformation $T_{ECF}$.} \label{cf-fig}
 \end{center}
 \end{figure}

The map $T_{BCF}$ has a neutral fixed point at $x=0$, as $T_{BCF}(0)=0$, $T'_{BCF}(0)=1$, and $T'_{BCF}(x)-1\sim 2x$ as $x\to 0^+$. As a consequence, $T_{BCF}$ preserves the infinite measure $\mu$ with density $k/x$ for $k>0$ (see \cite{renyi}), moreover it is conservative and ergodic on $([0,1],\mu)$. We will choose $k=1/\log 2$ so that setting $E:=[1/2,1]$ we have $\mu(E)=1$. For the following we define $I_n:=((n-1)/n, n/(n+1))$, $n\in \N$, hence $E$ is the closure of $\cup_{n\ge 2} I_n$. It is a classical result that every irrational $x\in [0,1]$ has a unique expansion of the form
\begin{equation} \label{coefficients-BCF}
x = 1 - \cfrac{1}{d_1-\cfrac{1}{d_2 - \dots}},
\end{equation}
with $d_j \in \N_{\ge 2}$ for all $j\ge 1$, and $d_j = k$ if and only if $T_{BCF}^{j-1}(x)\in I_{k-1}$ or equivalently if and only if
\[
\left\lfloor\frac{1}{1-T_{BCF}^{j-1}(x)}\right\rfloor = k-1
\]
(see, e.g., \cite{ik-book}).

Therefore, the partial sums of the coefficients of the backward continued fraction expansion of an irrational $x\in [0,1]$ are given by the Birkhoff partial sums of the observable
\begin{equation} \label{observ-bcf}
g:[0,1]\to \N, \quad g(x)=\left\lfloor\frac{1}{1-x}\right\rfloor +1
\end{equation}
for the transformation $T_{BCF}$. Note that $g$ is unbounded and non-integrable on $E$. In Section \ref{sec: proofs cf} we apply Theorem \ref{thm:main1} to $g$ to prove the following pointwise asymptotic behaviour. 

\begin{theorem}\label{cor:BCF}
For a.e.\ $x\in [0,1]$, the coefficients $\{d_j(x)\}$ of the expansion \eqref{coefficients-BCF} satisfy
\[
 \lim_{N\to \infty}\, \frac 1N\, \Bigg( \sum_{j=1}^{N+m(N,E,x)}\, d_j(x) - \max \left\{ 2\, m(N,E,x),\max_{1\leq k\leq N+m(N,E,x)} d_k(x) \right\} \Bigg) = 3.
\]
\end{theorem}

We recall that Aaronson \cite{aa-bcf} showed that $N^{-1}\sum_{j=1}^N d_j(x)$ converges in measure to 3, and Aaronson and Nakada \cite{aar-nakada} proved that the sequence $(N^{-1}\sum_{j=1}^N d_j(x))_N$ has ``maximal'' oscillations for a.e. $x$, that is the liminf is 2 and the limsup is $\infty$. For other stochastic properties of these sums we refer the reader to \cite{takahasi}.

Furthermore, we want to mention that it would be easily possible to apply Theorems \ref{thm:h-slow} and \ref{thm:h-fast} if we look at $h(d_j)$ instead of $d_j$, where $\lim_{n\to\infty} h(n)/n=0$ or $\lim_{n\to\infty} h(n)/n=\infty$ respectively.

In addition, for the backward continued fraction expansion we prove a result for rapidly increasing observables which needs additional techniques and is not applicable in the general framework of Theorem \ref{thm:h-fast}.

\begin{proposition}\label{prop:bcf-h-fast}
Let $\ell:\R\to \R$ be such that $\ell(t) \sim t^s$ with $s>1$, and $g:[0,1]\to \N$ be as in \eqref{observ-bcf}. 
Then for $u>1$ and $\mu$-a.e. $x\in X$ we have
\[
\lim_{N\to \infty}\, \frac{S^{(\log\log N)^u}_{N+m(N,E,x)}(\ell\circ g)(x)}{\gamma_N} = 1,
\]
where 
\[
\gamma_N\sim \frac{(\log 2)^s}{s-1}\, \left(\frac{N}{\log N}\right)^s\, (\log\log N)^{(1-s)u}.
\]
\end{proposition}

\begin{remark}
 It is clear that in this case Theorem \ref{thm:h-fast} can not be applied as $W=\infty$ and we thus need intermediate trimming. 
 It is also worth mentioning that it would also be possible to generalize the above proposition to functions where $\ell(t)=t^sL(t)$ with $L$ a slowly varying function. However, we want to mainly emphasize here that in the particular situation where the transfer operator of the induced map $T_{_E}$ has a spectral gap, we can even apply methods for intermediate trimming. 
\end{remark}

\begin{remark}
 It would be interesting to extend the results on this example in the following way: 
 In \cite{kkv17,brs20} a random continued fraction transformation was considered, i.e. one chooses randomly to use the regular continued fraction transformation or the backward continued fraction transformation. The question arises if the almost sure limit results just proven here can also be extended to random continued fractions and if in that case the term $m(N,E,x)$ still plays a role.
 
 On the other hand it is worth noticing that the backward continued fraction can be seen as an $\alpha$-continued fraction as studied in \cite{nak-nat-trim} for $\alpha=0$. 
 For the regular continued fraction $\alpha$ equals $1$. The results in \cite{diamond_estimates_1986} have been 
 generalized in \cite{nak-nat-trim} for $\alpha$-continued fractions for $\alpha\in [1/2,1]$. For those parameters of $\alpha$ the absolutely continuous invariant measure is finite. However, for arbitrary $\alpha\in [1/2,1]$ they show weaker mixing properties than the regular continued fraction expansion. Thus, it would be interesting to study the case $\alpha\in (0,1/2)$. Those continued fraction expansions have an infinite invariant measure and also their mixing properties might be worse than those of the regular continued fractions.
\end{remark} 
\end{example}

\begin{example}[Even-Integer Continued Fractions] \label{ex:ecf}
The backward continued fraction expansion can be considered as an example of a large class of continued fraction algorithms introduced in \cite{flipped} which are obtained by ``flipping'' the Gauss map on a subset of $[0,1]$. Another interesting example of this class is the \emph{even-integer continued fraction} algorithm first studied in \cite{schw-even, schw-even-2}. Accordingly, all irrational $x\in [0,1]$ have a unique infinite expansion of the form
\begin{equation} \label{coefficients-ECF}
x = \cfrac{1}{2h_1 + \cfrac{\eps_1}{2h_2+ \cfrac{\eps_2}{2h_3 + \dots}}}
\end{equation}
with $h_j \in \N$ and $\eps_j \in \{-1, +1\}$ for all $j\ge 1$. This expansion can be obtained by looking at the \emph{Even-Integer Continued Fraction} transformation $T_{ECF}:[0,1]\to[0,1]$, defined by 
\[
T_{ECF}(x)=\left\{ \begin{array}{ll} 1-\left\{ \frac 1x\right\}, & \text{if $x\in \bigcup_{k\ge 1} \left( \frac{1}{2k},\frac{1}{2k-1} \right)$};\\[0.2cm] \left\{ \frac 1x\right\}, & \text{otherwise}
\end{array} \right..
\]
See Figure \ref{cf-fig}-(b). As for $T_{BCF}$, the map $T_{ECF}$ has a neutral fixed point at $x=1$ and preserves an infinite measure. 

To study the sums of the $h_j$'s in the expansion \eqref{coefficients-ECF}, notice that the map $T_{ECF}$ may also be defined by letting $T_{ECF}(x)= | 1/x - 2k|$ with $2k$ being the nearest even integer to $1/x$.
Furthermore, $h_j(x)=k$ if and only if $T_{ECF}^{j-1}(x)\in I_{2k-1}\cup I_{2k}$ where $I_k=(1/(k+1), 1/k)$ or in other words $2h_j(x)= (g \circ T_{ECF}^{j-1})(x)$, where 
\begin{align}
g:[0,1]\to\N, \qquad g(x)= 2 \left\lfloor \frac{1}{2x}+\frac 12\right\rfloor\label{observ-ecf tilde}
\end{align}
and $2h_j(x) + \frac 12(\eps_j(x)-1)= (\tilde{g}\circ T_{ECF}^{j-1})(x)$
with 
\begin{align}
\tilde{g}:[0,1]\to\N, \qquad \tilde{g}(x)=\left\lfloor \frac{1}{x}\right\rfloor.\label{observ-ecf}
\end{align}

In Section \ref{sec: proofs cf} we prove the following result.
\begin{theorem}\label{cor:ECF}
For a.e. $x\in [0,1]$, the coefficients $\{h_j(x)\}$ of the expansion \eqref{coefficients-ECF} satisfy
\[
 \lim_{N\to \infty}\, \frac 1N\, \Bigg[ \sum_{j=1}^{N+m(N,E,x)}\, 2h_j(x)  - \max \left\{ 2 m(N,E,x),\max_{1\leq k\leq N+m(N,E,x)} 2h_k(x) \right\} \Bigg] = 3.
\]
The same result holds with $2h_j(x) + \frac 12(\eps_j(x)-1)$ instead of $2h_j(x)$.
\end{theorem}
\end{example}

\section{Proofs of the general results}\label{sec: proof gen thm}

Before we start with the proof of Theorem \ref{thm:main1} we introduce some additional notation and auxiliary lemmas. Let $R_{_{E,N}}(x)$ be the \emph{number of visits to $E$ up to time $N$} along the orbit of a point $x$, that is
\[
R_{_{E,N}}(x) := \sum_{k=1}^{N+1}\, (\UU_{_E}\circ T^{k-1})(x),
\]
and for an observable $g$ let the induced observable $g^E$ defined on the full measure set of points $x\in E$ with finite first return time to be
\begin{equation} \label{induced-obs}
g^E(x) := \sum_{k=1}^{\varphi_{_E}(x)}\, (g\circ T^{k-1})(x).
\end{equation}
The induced observable simply gives the contribution to the Birkhoff sums $S_Ng$ of the excursions out of $E$ for the orbit of a point $x$. For $\mu$-a.e. $x\in E$ let's define the time $\tau_{_{E,x}}(N)$ of the $N$-th return to $E$, which corresponds to the Birkhoff sum of $\varphi_{_E}$ for the system $(E,T_{_E})$. We have
\[
\tau_{_{E,x}}(N) := \sum_{k=1}^N\, (\varphi_{_E} \circ T_{_E}^{k-1}) (x).
\]

For simplicity we now consider orbits starting in $E$, and putting together the introduced notations, we write that the last visit to $E$ for $x\in E$ up to time $N$ is $T_{_E}^{R_{_{E,N}}(x)-1}(x)$ and the time of this last visit is then given by $\tau_{_{E,x}}(R_{_{E,N}}(x)-1)\le N$. Then the following relation follows for $\mu$-a.e. $x\in E$
\[
S_Ng(x) = \sum_{n=1}^{R_{_{E,N}}(x) -1}\, (g^E \circ T^{n-1}_{_E})(x) + \sum_{k=\tau_{_{E,x}}(R_{_{E,N}}(x)-1)+1}^{N}\, (g\circ T^{k-1})(x),
\]
with the standard convention that a summation vanishes if the upper index is smaller than the lower index. 

It is immediate to verify that for $\mu$-a.e.\ $x\in E$
\begin{equation}\label{relations}
m(N,E,x) = \max\set{(\varphi_{_E}\circ T_{_E}^{k-1})(x)\, :\, k=1,\dots, R_{_{E,N}}(x)}.
\end{equation}
We use the notation $R_{_{E,N,m}}(x) := R_{_{E,N+m(N,E,x)}}(x)$ for the number of visits to $E$ in the first $N+m(N,E,x)$ steps of the orbit of $x$. 

The next result is proved in \cite[Lemma 4.4]{no1}.
\begin{lemma}[\cite{no1}] \label{useful-3}
Let $E$ induce \mixset\ and recall the definition of the sequence $(\alpha(n))$ in \eqref{alfa}. Then for $\mu$-a.e.\ $x\in E$ we have
\begin{equation}\label{fine-1}
\lim_{N\to \infty}\, \frac{R_{_{E,N,m}}(x)}{\alpha(N)} = 1
\end{equation}
\end{lemma} 

Moreover, we need a general result about convergence of the trimmed sums of a sequence of random variables. The following is adapted from \cite{aar-nakada} (see also \cite[Appendix A]{no1}).

\begin{lemma}\label{lemma-aar-nakada}
Let $(Y_n)$ be a sequence of non-negative, identically distributed random variables on a probability space $(\Omega,\PP)$ which is $\psi$-mixing with the coefficient $\psi(n)$ fulfilling $\sum_{n\ge 1} \psi(n)/n < \infty$. Let $F$ be the distribution function of the $(Y_n)$, that is $F(y) = \PP (Y_1\le y)$, and let for some $y_0>1$
\begin{align}
 W\coloneqq \inf \left\{r\in\mathbb{N}\colon \int_{y_0}^{\infty}\left(\frac{y\left(1-F(y)\right)}{\int_{y_0}^y\left(1-F(t)\right)\mathrm{d}t}\right)^{r+1} \frac{1}{y}\, \mathrm{d}y<\infty\right\},
\label{eq: cond AN}
\end{align} 
with the standard convention that $W=\infty$ if such an $r$ does not exist. Then there exists a sequence $(b(n))$ such that for $\PP$-a.e. $\omega \in \Omega$
\begin{align*}
 \lim_{N\to \infty}\frac{S_N^W(\omega)}{b(N)}=1.
\end{align*}
If we set $a(y):=y/ \int_{y_0}^y (1-F(t))\mathrm{d}t$, then $b(n)$ can be set as the asymptotic inverse function of $a(n)$.

Furthermore, if we denote by $M^r_N(\omega)$ the $r$-th maximum in $\{Y_1(\omega),\ldots, Y_N(\omega)\}$ then we have
\[
 \lim_{N\to\infty}\frac{M^r_N(\omega)}{b(N)}=0 \quad \text{for $\PP$-a.e. $\omega\in \Omega$,} 
\]
for all $r> W$. (Here, we denote by the $r$-th maximum the maximal element that remains after removing the $r-1$ largest entries.)
\end{lemma}

Finally, we state a version of the Borel-Cantelli lemma which is applicable to $\psi$-mixing sequences:
\begin{lemma}[{\cite[Lemma 6]{BClemma}}]\label{lemma: BC improved}
 Let $A_n$ be a sequence of events such that $\mathbbm{1}_{A_n}$ is a $\psi$-mixing sequence of random variables. 
 If $\sum_{n=1}^{\infty}\mu(A_n)=\infty$, then $\mu(A_n\text{ i.o.})=1$. 
\end{lemma}
\begin{remark}
In \cite{BClemma} the notion of $\psi$-mixing is referred to as *-mixing.
\end{remark}

We are now ready to give the proof of the main results.

\subsection{Proof of Theorem \ref{thm:main1}} \label{sec:proof-main}
Let $g:X\to \R_{\ge 0}$ be a measurable observable satisfying assumptions (ii) and (iii) and bounded on $X\setminus E$. We may also assume that $\sup_{_{X\setminus E}} g \le \inf_{_E} g$. If this is not the case, we can consider $\tilde g := g + (\sup_{_{X\setminus E}} g)\, \UU_{_E}$ and prove \eqref{eq:main1} for $\tilde g$. Since 
\[
S_{N+m(N,E,x)}\tilde g(x) = S_{N+m(N,E,x)}g(x) + \left(\sup_{X\setminus E} g\right)\, R_{_{E,N,m}}(x),
\]
by Lemma \ref{useful-3} and \eqref{alfa}, which implies $\alpha(N)=o(N)$, we obtain \eqref{eq:main1} for $g$.

The first step is to split the observable $g$ into two summands which can be studied independently. Let $c:= \sup_{_{X\setminus E}} g \le \inf_{_E} g$ and consider
\begin{equation}\label{g-reduced}
g_c(x) := \left\{ \begin{array}{ll} g(x), & \text{if $x\in X\setminus E$;} \\[0.2cm] c, &  \text{if $x\in E$.} \end{array} \right.
\end{equation}
Then $g_c$ is bounded on $X$ and its Birkhoff sums may be studied as in \cite{no1} when the needed assumptions are satisfied. Further, let
\begin{equation}\label{g-f}
f(x) := g(x) - g_c(x),\quad \forall\, x\in X
\end{equation}
then $f$ is a non-negative observable which vanishes on $X\setminus E$ and $f\not\in L^1(E,\mu)$. Its Birkhoff sums may be studied by looking only at the induced map $T_{_E}$, in fact the induced observable $f^E$ defined as in \eqref{induced-obs} satisfies 
\begin{equation}\label{induced-obs-reduced}
f^E(x) = f(x),\quad \forall\, x\in E,
\end{equation}
so that for all $N$ and a.e. $x\in E$
\begin{equation}\label{induced-obs-sums}
\sum_{n=1}^N\, (f\circ T^{n-1})(x) = \sum_{n=1}^{R_{_{E,N}}(x)-1}\, (f^E\circ T_{_E}^{n-1})(x).
\end{equation}
We use the linearity of the Birkhoff sums to write
\begin{equation} \label{splitting}
S_N g(x) = S_N g_c (x) + S_N f(x),
\end{equation}
so that we can study the asymptotic behaviour of the two sums separately. 

Assuming now that $g$ satisfies assumption (iv), we obtain $g_c \equiv c$ on $X$, so that $S_N g_c (x) = c\, N$ for all $x\in X$. Hence (iv) is useful to reduce the proof to the study of the Birkhoff sums of $f$.  

Now, we aim to apply Lemma \ref{lemma-aar-nakada} on the sequence $(f^E\circ T_{_E}^{n-1})$.
By condition (iii), the function $f$ is piecewise constant on the partition $\PPP_{_E}$, and by the assumptions, the set $E$ induces rapid $\psi$-mixing (see Definition \ref{def-set-induce}), therefore the sequence $(f^E\circ T_{_E}^{n-1})$ is $\psi$-mixing with the coefficient $\psi(n)$ fulfilling $\sum_{n\ge 1}\, \psi(n)/n <\infty$. 
Furthermore, we have 
\[
 1-F(n) = \mu(f^E>n)=\mu(f>n). 
\]
So, a straightforward calculation using (i) and (ii) shows that $W=1$ in \eqref{eq: cond AN}. Hence, Lemma \ref{lemma-aar-nakada} is applicable. Furthermore, the asymptotic inverse of the sequence $(b(n))$ of Lemma \ref{lemma-aar-nakada} is
\begin{equation}\label{eq:a-ylog1}
 a(y)=\frac{y}{\int_1^y (1-F(t))\mathrm{d}t}
     \sim \frac{y}{\int_1^y \kappa \mu(A_{>\lfloor t\rfloor}) \mathrm{d}t}
     \sim \frac{y }{\kappa\, w_{\lfloor y\rfloor}(E)}
     \sim \frac{\alpha(y)}{\kappa}.
\end{equation}
Thus, since $R_{E,N,m}(x)\to\infty$ as $N\to \infty$, we obtain that for $\mu$-a.e. $x\in X$ 
\[
\lim_{N\to \infty}\, \frac{\sum_{n=1}^{R_{E,N,m}(x)-1}\, (f^E\circ T_{_E}^{n-1})(x)- \max_{1\leq k\leq R_{E,N,m}(x)-1} (f^E\circ T_{_E}^{k-1})(x)}{b(R_{E,N,m}(x))}\,  = 1.
\]
Since by (i) the wandering rate $w_n(E)$ is slowly varying, we have that $a(y)$ is regularly varying with index $1$. Hence, it's asymptotic inverse function $b$ is also regularly varying with index $1$, see e.g.\ \cite[Prop.~1.5.12]{reg-var-book}. Then by Lemma \ref{useful-3} and \eqref{eq:a-ylog1} we obtain for $\mu$-a.e. $x\in X$ 
\begin{equation} \label{eq: asymp alpha}
b( R_{E,N,m}(x) ) \sim b( \alpha(N))   
\sim b(\kappa\, a(N))\sim \kappa\, N.
\end{equation}
Hence for $\mu$-a.e. $x\in X$ we have
\begin{equation} \label{eq:case-f1}
  \lim_{N\to \infty}\, \frac{\sum_{n=1}^{R_{E,N,m}(x)-1}\, (f^E\circ T_{_E}^{n-1})(x)-\max_{1\leq k\leq R_{E,N,m}(x)-1} (f^E\circ T_{_E}^{k-1})(x)}{N}\,  =  \kappa.
\end{equation}

Let's now consider the function $g$ and its splitting $g=g_c+f$. By the assumptions we have
\begin{align*}
S_{N+m(N,E,x)} g(x) &\, = c\,(N+m(N,E,x))+S_{N+m(N,E,x)} f(x)  \\ &\, = c\,(N+m(N,E,x)) +\sum_{n=1}^{R_{E,N,m}(x)-1}\, (f^E\circ T_{_E}^{n-1})(x),
\end{align*}
and
\begin{align}
\max_{1\leq k\leq R_{E,N,m}(x)-1} (f^E\circ T_{_E}^{k-1})(x) = \max_{1\leq k\leq N+m(N,E,x)} (g\circ T^{k-1})(x) -c.\label{eq: max f g}
\end{align}
Finally, using \eqref{eq:case-f1} we have
\begin{align}
&\, \lim_{N\to \infty}\, \frac{S_{N+m(N,E,x)} g(x) -c\, m(N,E,x) - \max_{1\leq k\leq N+m(N,E,x)} (g\circ T^{k-1})(x) +c}{N}\notag \\
&\, = \lim_{N\to \infty}\, \frac{\sum_{n=1}^{R_{E,N,m}(x)-1}\, (f^E\circ T_{_E}^{n-1})(x) +c\, N - \max_{1\leq k\leq R_{E,N,m}(x)-1} (f^E\circ T_{_E}^{k-1})(x)}{N} = c+\kappa,\label{eq: SN+m=c+kappa}
\end{align}
and \eqref{eq:main1} is proved.\newline 

In the sequel, we prove \eqref{eq:main2} under conditions (i)-(iv) and \eqref{eq:cond1}. As before, it is enough to study the function $f$ defined in \eqref{g-f}. We first show that for a monotonic sequence $( p_N)$ tending to infinity we have that\footnote{We use ``i.m.'' to denote ``infinitely many''.}  
\begin{align}
\mu\left(\bigcup_{k, \ell \in \{1,\ldots, N\}} \left\{f \circ T_{_E}^{k-1}> p_N\right\} \cap \left\{\varphi_{_E} \circ T_{_E}^{\ell-1}>p_N\right\}  \text{ for i.m. }N\in\N\right)=0\label{eq: cor equiv 1}
 \end{align}
holds if
\begin{align}
  \sum_{N=1}^{\infty}\, N\, \mu(g >p_N)^2<\infty.\label{eq: cor equiv 2}
\end{align}

We follow the ideas of \cite[Lemma 1.2]{aar-nakada}. First of all, since $E$ is \mixset, for $k\neq \ell$ we have 
\begin{align*}
 & \mu\left(\left\{f \circ T_{_E}^{k-1}>p_N \right\} \cap \left\{\varphi_{_E} \circ T_{_E}^{\ell-1}>p_N\right\}\right)\\
  &\leq (1+\psi(|k-\ell|))\, \mu\left(f \circ T_{_E}^{k-1}>p_N\right)\, \mu\left( \varphi_{_E} \circ T_{_E}^{\ell-1}>p_N\right)\\
  &\leq \kappa'\, \mu\left(f >p_N\right)^2\,, 
\end{align*}
where we have used (ii) and have $\kappa':=(1+\psi(1))/\kappa$. Using additionally $\mu\left(f >p_N\right)\sim \mu\left(g >p_N\right)$ for $p_N\to\infty$ and \eqref{eq:cond1} there exists $\kappa''$ such that
\begin{align*}
 & \mu\left(\bigcup_{k, \ell \in \{1,\ldots, N\}} \left\{f \circ T_{_E}^{k-1}>p_N\right\} \cap \left\{\varphi_{_E} \circ T_{_E}^{\ell-1}>p_N\right\}\right)\\
 &\qquad\leq \sum_{k=1}^{N} \sum_{\ell=1}^N \mu \left(\left\{f \circ T_{_E}^{k-1}>p_N \right\} \cap \left\{\varphi_{_E} \circ T_{_E}^{\ell-1}>p_N \right\} \right)
 \leq \kappa''\, N^2\, \mu\left(g >p_N\right)^2.
\end{align*}
The remaining proof of the equivalence between \eqref{eq: cor equiv 1} and \eqref{eq: cor equiv 2} follows then in the same way as in the proof of \cite[Lemma 1.2]{aar-nakada} using the methods of the direct part of \cite[Lemma 3]{mori76}.

In the next steps we will show that 
\eqref{eq: cor equiv 2} holds for $p_N=\beta(N)$ with $\beta$ the asymptotic inverse of $\alpha$ as in \eqref{alfa}. First, we note that by (ii),
\begin{align*}
 \sum_{n\geq 1}\frac{n (\mu(A_{>n}))^2}{w_n(E)^2}
 &\asymp \sum_{n\geq 1}\frac{\alpha(n)^2 (\mu(A_{>n}))^2}{n}
 \asymp \int_0^{\infty}\frac{\alpha(t )^2 (\mu(A_{>t}))^2}{t}\,\mathrm{d}t
 \asymp \int_0^{\infty}\frac{\alpha(t )^2 \mu(g>t)^2}{t}\,\mathrm{d}t.
\end{align*}
Here and in the following we slightly abuse notation by setting $\gamma(t)=\gamma(\lfloor t\rfloor)\{t\}+\gamma(\lfloor t\rfloor+1)\,(1-\{t\})$, where $\gamma$ denotes any function in this and the following calculations.
Next, we use substitution with $y(t)=\alpha(t)=t/w_t(E)$. We have 
\begin{align*}
 \frac{\mathrm{d}y(t)}{\mathrm{d}t}=\frac{w_t(E)+t\mu(A_{>t})}{w_t(E)^2}\asymp \frac{1}{w_t(E)},
\end{align*}
since $\mu(A_{>t})t\leq w_t(E)$. 

Then we have
\[
 \sum_{n\geq 1}\frac{n (\mu(A_{>n}))^2}{w_n(E)^2}
 \asymp \int_0^{\infty}\frac{y^2 \mu(g>\beta(y))^2 w_{\beta(y)}(E) }{\beta(y)}\,\mathrm{d}y
 \asymp \int_0^{\infty} y\, \mu(g>\beta(y))^2 \,\mathrm{d}y
\]
and thus \eqref{eq: cond level sets} implies 
\begin{equation}\label{eq: sum with C}
 \sum_{n\geq 1} n\, \mu(g> \beta(n))^2<\infty.
\end{equation}

Moreover, we note that by an analogous argumentation as in \cite[Lemma 1.2]{aar-nakada}, noting that $\beta$ is regularly varying, we have that \eqref{eq: sum with C} implies
$\sum_{n\geq 1} n\, \mu(g>\epsilon \beta(n))^2<\infty$, for all $\epsilon>0$. 

Hence, we can conclude that \eqref{eq: cor equiv 1} holds with $ p_N=\epsilon \beta(N)$ for any $\epsilon>0$. Using that $\beta$ is the asymptotic inverse of $\alpha$ and Lemma \ref{useful-3} this is equivalent to 
 \begin{align}
  \mu\left(\bigcup_{k, \ell \in \{1,\ldots, R_{_{E,N,m}}(x)\}} \left\{ f \circ T_{_E}^{k-1}>  \epsilon N\right\} \cap \left\{\varphi_{_E} \circ T_{_E}^{\ell-1}> \epsilon N\right\}  \text{ for i.m. }N\in\N\right)=0.\label{eq: cor last step}
 \end{align}
Here, we have additionally used that $\beta$ as the asymptotic inverse of a regularly varying function has to be regulaly varying itself. 
 
Then we notice that $m(N,E,x) \leq  \max\set{(\varphi_{_E}\circ T_{_E}^{k-1})(x)\, :\, k=1,\dots, R_{_{E,N,m}}(x)}$ holds
by \eqref{relations}. Combining this with \eqref{eq: cor last step}, \eqref{eq: SN+m=c+kappa} and the fact that $f=f^E$ we obtain \eqref{eq:main2}.\newline
 
Finally, we give the proof of \eqref{eq-cfr}. By \eqref{g-f} and \eqref{eq: max f g} we have 
 \begin{align}
  \max_{1\leq k\leq N+m(N,E,x)} (g\circ T^{k-1})(x)
  =\max_{1\leq k\leq R_{E,N,m}(x)-1} (f^E\circ T_{_E}^{k-1})(x) +c.\label{eq:max g max fE}
 \end{align}

By Aaronson's theorem, \cite{aa-book}, the fact that $f^E$ is non-integrable, the fact that $R_{E,N,m}(x)$ is an increasing function with $0$ or $1$ increments, and \eqref{eq: asymp alpha} we have 
\begin{align*}
 \limsup_{N\to\infty}\frac{ \sum_{k=1}^{N} (f^E\circ T_{_E}^{k-1})(x)}{b(N)}
  \,=\,\limsup_{N\to\infty}\frac{ \sum_{k=1}^{R_{E,N,m}(x)-1} (f^E\circ T_{_E}^{k-1})(x)}{N}\in \{0,\infty\}.
\end{align*}
However, by \eqref{eq:case-f1}, it is clear that the above limit has to equal $\infty$ and we obtain using \eqref{eq:max g max fE} that
\[
\limsup_{N\to\infty}\frac{ \max_{1\leq k\leq N+m(N,E,x)} (g\circ T^{k-1})(x)}{N}
=\limsup_{N\to\infty}\frac{ \max_{1\leq k\leq R_{E,N,m}(x)-1} (f^E\circ T_{_E}^{k-1})(x)}{N}
= \infty.
\] 

On the other hand, we have  
\begin{align*}
 m(N, E, x) = \max \left\{(\varphi_{_E} \circ T_{_E}^{k-1})(x) : k = 1, \ldots , R_{E,N}(x)\right\}, 
\end{align*}
see \cite[p. 5558]{no1}.
Furthermore, from \cite[Lemma 4.3]{no1} we have for $\mu$-a.e. $x\in X$ that
 \begin{align*}
  \lim_{N\to\infty}\frac{\sum_{k=1}^{N}(\varphi_{_E}\circ T_{_E}^{k-1})(x)-\max_{1\leq \ell \leq N} (\varphi_{_E}\circ T_{_E}^{\ell-1})(x) }{\beta(N)}=1,
 \end{align*}
where $\beta$ is the asymptotic inverse of $\alpha$ as in \eqref{alfa}.  
Hence, Aaronson's theorem, \cite{aa-book}, and the fact that $\varphi_E$ is non-integrable implies for $\mu$-a.e.\ $x\in X$ that
\begin{align*}
 \limsup_{N\to\infty}\frac{\max_{1\leq \ell \leq N} \varphi_{_E}\circ T_{_E}^{\ell-1}(x) }{\beta(N)}
 =\limsup_{N\to\infty}\frac{\varphi_{_E}\circ T_{_E}^{N-1}(x) }{\beta(N)}
 =\infty.
\end{align*}
Moreover, since $R_{E,N}$ is an increasing function with increments at most $1$ and using \cite[Proof of Thm.~2.4]{no1} it satisfies $\liminf_{N\to \infty} d(R_{E,N}(x))/N\geq 1$ for $\mu$-a.e. $x\in X$, we also obtain 
\begin{align*}
 \limsup_{N\to\infty} \frac{m(N, E, x)}{N}
 \geq \limsup_{N\to\infty}\frac{\varphi_{_E}\circ T_{_E}^{R_{E,N}(x)-1}(x)}{N}=\infty. 
\end{align*}
This concludes the proof of Theorem \ref{thm:main1}.

\subsection{Proofs for the theorems with flattened or increased observables}

\begin{proof}[Proof of Theorem \ref{thm:h-slow}]
We argue as in the proof of Theorem \ref{thm:main1}. First, we write $g(x)=c+f(x)$ with $f$ non-negative, $f\not\in L^1(E,\mu)$, and vanishing on $X\setminus E$. Therefore, $S_N g = cN + S_N f$ and as in \eqref{induced-obs-reduced}-\eqref{induced-obs-sums}
\[
S_N f(x) = \sum_{n=1}^{R_{_{E,N}}(x)-1}\, (f \circ T_{_E}^{n-1})(x).
\]
Let us assume (i), (iii), and (iv) of Theorem \ref{thm:main1}, and (iia). Then, we apply Lemma \ref{lemma-aar-nakada} to the sequence $(f^E \circ T_{_E}^{n-1})$ with
\[
1- F(n) = \mu(f > n) \sim \mu(g>n) = o(\mu(A_{>n})).
\]
It follows that there exists a sequence $(b(n))$ with asymptotic inverse $a(y):=y/ \int_{y_0}^y (1-F(t))\mathrm{d}t$ such that for $\mu$-a.e. $x$
\[
\lim_{N\to \infty}\, \frac{\sum_{n=1}^{R_{_{E,N}}(x)-1}\, (f \circ T_{_E}^{n-1})(x) - \max_{1\le k\le R_{_{E,N}}(x)-1}\, (f \circ T_{_E}^{k-1})(x)}{b(R_{_{E,N}}(x)-1)} = 1,
\] 
and the sequence $(\alpha(n))$ defined in \eqref{alfa} satisfies $\alpha(n)=o(a(n))$.

We now recall from \cite[eq. (34)]{no1} that given for $x\in E$ the quantity
\[
w(N,E,x) := \max \set{ k\ge 1\, :\, \exists\, \ell \in \{1,\dots,N-k\} \text{ s.t. } T^{\ell+j}(x) \not\in E,\, \forall\, j=0,\dots,k-1}. 
\]
which measures the length of an excursion which does not return to $E$ up to time $N$, the following asymptotic behaviour is true
\begin{equation}\label{fine-2}
\lim_{N\to \infty}\, \frac{R_{_{E,N}}(x)}{\alpha(N-w(N,E,x))} = 1.
\end{equation}
Therefore, we obtain
\[
b(R_{_{E,N}}(x)) \sim b(\alpha(N-w(N,E,x))) = o(N).
\]

Finally, since
\[
\max_{1\leq k\leq R_{E,N}(x)-1} (f \circ T_{_E}^{k-1})(x) = \max_{1\leq k\leq N} (g\circ T^{k-1})(x) -c,
\]
we have
\begin{align*}
&\, \lim_{N\to \infty}\, \frac{S_{N} g(x) - \max_{1\leq k\leq N} ( g\circ T^{k-1})(x) +c}{N} \\
&\, = \lim_{N\to \infty}\, \frac{\sum_{n=1}^{R_{E,N}(x)-1}\, ( f\circ T_{_E}^{n-1})(x) +c\, (N-1) - \max_{1\le k\le R_{_{E,N}}(x)}\, ( f \circ T_{_E}^{k-1})(x)}{N} = c,
\end{align*}
and \eqref{slow-first-r} is proved.

For the second part of the proof, we need the following lemma:
\begin{lemma}[Borel-Bernstein type lemma]\label{lemma: Borel-Bernstein}
Let $T : X \to X$ be a conservative, ergodic, measure-preserving transformation of
the measure space $(X, \mathcal{B}, \mu)$ with $\mu$ a $\sigma$-finite measure with $\mu(X) = \infty$, and let $E \in\mathcal{B}$ with $\mu(E) = 1$ be a set which induces rapid $\psi$-mixing.

For $\ell:E\to\mathbb{R}$ such that $\ell$ is constant on each of the sets of the partition $\mathcal{P}_{_E}$ and for a sequence $( \delta_n)$, we have 
\[
\mu(\ell\circ T_{_E}^{n-1}>\delta_n\ \text{for i.m. } n\in\N)= \left\{
\begin{array}{ll} 0, & \text{if }\sum_{n=1}^{\infty}\mu(\ell>\delta_n)<\infty;\\ 1,&\text{else.} \end{array} \right.
\]
\end{lemma}
 
\noindent \emph{Proof of Lemma \ref{lemma: Borel-Bernstein}.}
Since $E$ induces rapid $\psi$-mixing and $\ell$ is constant on each set in $\mathcal{P}_{_E}$, we see that the  Borel-Cantelli Lemma \ref{lemma: BC improved} applies
and we can show that 
\begin{equation} \label{Eq_RBCLa}
\mu(\ell\circ T_{_E}^{n-1}>\delta_n\ \text{for i.m. } n\in\N)=1 \quad \text{iff} \quad \sum_{n\geq1}\mu(\ell\circ T_{_E}^{n-1}>\delta_n)=\infty.
\end{equation}

Since $\mu$ is $T_{_E}$-invariant we have $\mu(\ell\circ T_{_E}^{n-1}>\delta_n)=\mu(\ell>\delta_n)$, so that divergence of the right-hand series in \eqref{Eq_RBCLa} is equivalent to that of $\sum_{n\geq1}\mu(\ell>\delta_n)$. 
\qed

\vskip 0.3cm
To prove \eqref{slow-second-r} it remains to show that using (ii\~{a}) instead of (iia), we have
\[
\max_{1\le k\le N}\, (g \circ T^{k-1})(x) = o(N),\quad \text{for $\mu$-a.e. $x$.}
\] 
We note that for a.e.\ $x\in X$ we have by \eqref{fine-2}
\begin{align*}
\max_{1\le k\le N}\, (g \circ T^{k-1})(x)
 = \max_{1\leq k\leq R_{E,N}(x)-1} (f \circ T_{_E}^{k-1})(x)
 \leq \max_{1\leq k\leq \alpha(N)} (f \circ T_{_E}^{k-1})(x)
\end{align*}
Moreover, we have 
\begin{align*}
 \MoveEqLeft\mu\left(\max_{1\leq k\leq \alpha(N)} (f \circ T_{_E}^{k-1})(x)>\epsilon N \text{ for i.m. }N\in\N \right)\\
 &\leq \mu\left(\max_{1\leq k\leq N} (f \circ T_{_E}^{k-1})(x)>\frac{\epsilon}{2} \beta(N) \text{ for i.m. }N\in\N \right).
\end{align*}
To show that the above term equals $0$ for any $\epsilon>0$ we will make use of Lemma \ref{lemma: Borel-Bernstein} with $\ell=f$ and $\delta_n=\tilde{\epsilon}\beta(n)=(\epsilon/2) \beta(n)$. In this case we have by (ii\~{a}) that $\sum_{n=1}^{\infty} \mu(f >\tilde{\epsilon} \beta(n)) <\infty$, hence by Lemma \ref{lemma: Borel-Bernstein} we have $\mu(f\circ T_{_E}^{n-1} >\tilde{\epsilon} \beta(n) \text{ i.o.})=0$. Since $\delta_n=\tilde{\eps} \beta(n)$ is monotonically increasing, we can conclude that $\mu(\max_{1\le k\le n} f\circ T_{_E}^{k-1} >\eps \beta(n)\ \text{ i.o.})=0$ giving the statement of the theorem.
\end{proof}

\begin{proof}[Proof of Theorem \ref{thm:h-fast}]
Arguing as in the proof of Theorem \ref{thm:h-slow}, we find
\[
S_{N+m(N,E,x)} g(x) = c\Big(N+m(N,E,x)\Big) + \sum_{n=1}^{R_{_{E,N,m}}(x)-1}\, (f \circ T_{_E}^{n-1})(x).
\]
Then, we apply Lemma \ref{lemma-aar-nakada} to the sequence $(f \circ T_{_E}^{n-1})$ with $1- F(n) = \mu(f>n)$.
Here, we notice that $\mu(f>n)\sim \mu(g>n)$ and thus, $W$ in \eqref{eq: cond AN} does not change if we calculate it for $F$ being defined as $F(n)=1-\mu(f>n)$ or as $F(n)=1-\mu(g>n)$. We obtain that for $\mu$-a.e. $x$
\[
\lim_{N\to \infty}\, \frac{S^{W}_{R_{_{E,N,m}}(x)-1} f(x)}{b(R_{_{E,N,m}}(x)-1)} =1,
\]
where
\[
S_{R_{_{E,N,m}}(x)-1} f(x) := \sum_{n=1}^{R_{_{E,N,m}}(x)-1}\, ( f \circ T_{_E}^{n-1})(x),
\]
and $S^{W}_{R_{_{E,N,m}}(x)-1}$ is obtained from $S_{R_{_{E,N,m}}(x)-1}$ by trimming the largest $W$ entries (see \eqref{trimmed-birk-sum}). In addition, \eqref{fine-1} implies
\[
\lim_{N\to \infty}\, \frac{S^{W}_{R_{_{E,N,m}}(x)-1} f(x)}{b(\alpha(N))} =1.
\]
We conclude as before by writing
\[
S^{W}_{N+m(N,E,x)} g(x) - c\, m(N,E,x) = c\, N + S^{W}_{R_{_{E,N,m}}(x)-1} f (x) +c\, W_h
\]
and using that 
\[
\frac{b(\alpha(N))}{N} \to \infty
\]
since $a(n) = o(\alpha(N))$.
\end{proof}

\section{Proofs of the theorems for non-regular continued fraction expansions}\label{sec: proofs cf}
 
Given $T:X\to X$ a conservative ergodic measure-preserving transformation of the measure space $(X,\BB,\mu)$ and $E\in \BB$ a set with $\mu(E)=1$, we recall the \emph{hitting time function}
\[
h_{_E} : X \to \N \cup \{0\},\quad h_{_E}(x) := \inf \set{ k\ge 0\, :\, T^k(x) \in E}.
\]
Analogously to the level sets $A_n$ for $\varphi_E$ in \eqref{level-sets} we define a second family of level sets corresponding to $h_E$ given by
\[
E_n := \{ x\in X\, :\, h_{_E}(x) = n\}
\]
with $E_0=E$, for which $X = \bigsqcup_{k\ge 0}\, E_k$ up to zero measure sets. As a connection between these two we recall that $\mu(A_{>n}) = \mu(E_n)$ for all $n\ge 0$ (see \cite[Lemma 1]{zwei-notes}).

\subsection{Proofs for the Backward Continued Fraction transformation}
Given the Backward Continued Fraction transformation $T_{BCF}:[0,1]\to [0,1]$ as in Example \ref{ex:bcf}, we consider the set $E=[1/2,1]$ and the function $g:[0,1]\to \N$ defined in \eqref{observ-bcf}. We have previously observed that $g(T^{j-1}_{BCF}(x)) = d_j(x)$, the $j$-th coefficient in the backward continued fraction expansion \eqref{coefficients-BCF} of $x\in [0,1]$, hence
\[
\sum_{j=1}^{N+m(N,E,x)}\, d_j(x) = S_{N+m(N,E,x)} g(x),\quad \text{for a.e. $x\in [0,1]$.}
\]
To prove Theorem \ref{cor:BCF}, we will apply Theorem \ref{thm:main1} to the function $g$. We start with two auxiliary lemmas. First, we notice that the induced transformation on $E$ has countably many full slopes on each of the intervals $I_n=((n-1)/n,n/(n+1))$ with $n\ge 2$, see e.g.\ \cite[Lemma 8]{zwei}. Therefore, the partition $\PPP_{_E}$ induced by $T_{_E}$ is given by the sets $\{A_k \cap I_h\}$ with $k\ge 1$ and $h\ge 2$.

\begin{lemma}\label{lemma: psi mixing partition}
The set $E=[1/2,1]$ induces \mixset\ for the Backward Continued Fraction transformation $T_{BCF}$. 
\end{lemma}
\begin{proof}
It is enough to apply the results in \cite[Chap. 4]{aa-book} since the transformation $T_{BCF}$ is a $C^2$-Markov interval map with bounded distortion, namely $\sup_n\, (\sup_{I_n} |T''/(T')^2|)$ is finite. 
\end{proof}

\begin{lemma} \label{lem:aux1}
Let $E=[1/2,1]$ and $\varphi_{_E}$ the first return time function to $E$ for $T_{BCF}$. For $M,N\in\mathbb{N}$ we have
\begin{align*}
  \mu\left(\left\{g >N\right\} \cap \left\{\varphi_{_E} >M\right\}  \right)
  \asymp \mu\left(g >N\right) \mu \left(\varphi_{_E} >M \right),
\end{align*}
where $\mu$ is the $T_{BCF}$-invariant measure with density $1/(x \log 2)$.
\end{lemma}
\begin{proof}
First, the level sets of the hitting time functions $h_{_E}$ are $E_n= (1/(n+2),1/(n+1))$. Thus 
\begin{equation}\label{eq: En}
 \mu(A_{>n}) =  \mu(E_n) = \frac{1}{\log 2}\, \log \left( \frac{n+2}{n+1} \right) = \frac{1}{\log 2}\,\log\left(1+\frac{1}{n+1}\right)\sim \frac{1}{n\log 2}.
\end{equation}
Moreover, since $T_{BCF}$ is piecewise monotone and with full branches with respect to the partition $\{I_m\}$ of $[0,1]$ and $E=\cup_{m\ge 2} I_m$ up to zero-measure sets, one can easily calculate that the level sets of $\varphi_{_E}$ are given by
 \[
 A_n= T_{BCF}^{-1}(E_{n-1})\cap E = \bigcup_{m=2}^{\infty}\left(1-\frac{n+1}{mn+m+1}\, ,\, 1-\frac{n}{mn+1}\right), \quad \forall\, n\ge 1.
 \]
Therefore, 
 \[
 \left\{\varphi_{_E} >M\right\} = A_{>M} = \bigcup_{n\ge M+1} A_n = \bigcup_{m=2}^{\infty}\left(1-\frac{1}{m}\, ,\, 1-\frac{M+1}{m(M+1)+1}\right), \quad \forall\, M\ge 0.
 \]
 Then we recall that the function $g$ is locally constant on the partition $\{I_m\}$ with $g|_{I_m} \equiv m-1$.
 It follows that 
  \begin{align*}
  \left\{g > N \right\} \cap \left\{\varphi_{_E} >M\right\}
  =\bigcup_{m=N}^{\infty}\left(1-\frac{1}{m}\, ,\, 1-\frac{M+1}{m(M+1)+1}\right).
 \end{align*}
Finally, we use that $\mu$ is equivalent to the Lebesgue measure $\lambda$ on $E$. Since for $M,N$ big enough
 \[
 \lambda\left(g > N \right) = \sum_{m= N}^\infty\, \lambda(I_m) = \frac{1}{N},
 \]
 \begin{align*}
   \lambda\left( \varphi_{_E} >M\right) &\, = \sum_{m= 2}^\infty\, \lambda \left(1-\frac{1}{m}\, ,\, 1-\frac{M+1}{m(M+1)+1}\right)
    = \sum_{m= 2}^\infty\, \frac{1}{m^2(M+1)+m} \asymp \frac 1M,
\end{align*}
and 
 \begin{align*}
  \lambda\left( \left\{g > N \right\} \cap \left\{\varphi_{_E} >M\right\}\right)
  &= \sum_{m=N}^{\infty} \lambda\left( 1-\frac{1}{m}\, ,\, 1-\frac{M+1}{m(M+1)+1} \right)\\
  &= \sum_{m=N}^{\infty}\frac{1}{m^2(M+1)+m} \asymp \frac{1}{MN},
 \end{align*}
the statement of the lemma follows.
\end{proof}

\vskip 0.5cm

\noindent \emph{Proof of Theorem \ref{cor:BCF}.} 
First, by \eqref{eq: En} it follows that assumption (i) of Theorem \ref{thm:main1} is fulfilled. Then, the function $g:[0,1]\to \N$ given in \eqref{observ-bcf} clearly satisfies assumptions (iii) and (iv) of Theorem \ref{thm:main1} with $c=2$. Finally, using the invariant measure $\mu$ of the transformation $T_{BCF}$, which has density $1/(x \, \log 2)$, we obtain 
\begin{equation}\label{eq: asymp prop bwcf}
 \mu \left( g > n \right) = \sum_{m= n}^\infty\, \mu(I_m) =  \frac{1}{\log 2}\, \int_{\frac{n-1}{n}}^1\, \frac 1x\, \mathrm{d}x \sim \frac{1}{n\, \log 2} \sim \mu(A_{>n}),
 \end{equation}
hence, assumption (ii) of Theorem \ref{thm:main1} is satisfied with $\kappa= 1$. Thus, if we use that $g(T^{j-1}_{BCF}(x)) = d_j(x)$ for all $j\ge 1$ and a.e.\ $x\in [0,1]$, we obtain 
\[
 \lim_{N\to \infty}\, \frac 1N\, \Big( \sum_{j=1}^{N+m(N,E,x)}\, d_j(x) -  \max_{1\leq k\leq N+m(N,E,x)} d_k(x) - 2\, m(N,E,x) \Big) = 3.
\]
By Lemma \ref{lem:aux1}, also \eqref{eq:cond1} is fulfilled and we obtain the statement of the theorem. \qed

\begin{proof}[Proof of Proposition \ref{prop:bcf-h-fast}]
Arguing as in the proof of Theorem \ref{thm:h-fast} with $g$ defined as in \eqref{observ-bcf}, we set $f(x) = g(x)-2$ and we may write
\[
 S_{N +m(N,E,x)} g(x) = 
 2(N + m(N, E, x))+\tilde S_{N +m(N,E,x)} f(x)
\]
with the notation
\[
\tilde S_{N +m(N,E,x)} f(x) := \sum_{n=1}^{R_{_{E,N,m}}(x)}\, (f \circ T_{_E}^{n-1})(x).
\]
Now we cannot apply Lemma \ref{lemma-aar-nakada} to the sequence $(f \circ T_{_E}^{n-1})$. Instead we use results from \cite{kesseboehmer_strong_2019}, in particular Theorem 1.7 and its following Remark 1.8 and Proposition 1.12, as follows: 
\cite[Eq.~(10) and (11)]{kesseboehmer_strong_2019} can be easily verified for $f$ and thus \cite[Proposition 1.12]{kesseboehmer_strong_2019} can be applied to show that Property $\mathfrak{D}$ holds and \cite[Theorem 1.7]{kesseboehmer_strong_2019} is applicable. Hence, for any sequence $(r_n)$ such that 
\begin{equation}\label{eq: cond on rn}
\lim_{n\to \infty}\, \frac{r_n}{\log \log n} = \infty\quad\text{ and }\quad r_n=o(n),
\end{equation}
and for $\mu$-a.e. $x\in X$ we have, setting $\tilde r_N(x) := r_{{R_{E,N,m}(x)}}$,
\[
\lim_{N\to \infty}\, \frac{\tilde S_{N +m(N,E,x)}^{\tilde r_N(x)} f(x)}{b(R_{_{E,N,m}(x)})} =1,
\]
where we are using the notation for the intermediately trimmed sums (see \eqref{inter-trimmed-birk-sum}), and the norming sequence $b(n)$ satisfies
\[
b(n) \sim \frac{1}{s-1}\, r_n^{1-s} \, n^s 
\]  
as in \cite[Remark 1.7]{kesseboehmer_strong_2019}.

Finally, we use that $R_{_{E,N,m}}(x) \sim N \log 2 / \log N$ because of \eqref{fine-1} and \eqref{eq: asymp prop bwcf}, which gives
\[
 b(R_{_{E,N,m}}(x))\sim \frac{(\log 2)^s}{s-1}\, r_{N\log 2/ \log N}^{1-s} \, \left(\frac{N}{\log N}\right)^s.
\]
Setting now $r_n=(\log\log (n\log n/\log 2))^u$, noting that this is a possible choice fulfilling \eqref{eq: cond on rn} and noting finally that 
\[
N + m(N,E,x) = o \left( b(R_{_{E,N,m}}(x)) \right),
\]
gives the statement of the proposition.
\end{proof}

\subsection{Proofs for the Even-Integer Continued Fraction transformation}
The map $T_{ECF}$ preserves the infinite measure $\mu$ with density $k/(1-x^2)$ for $k>0$, see \cite{schw-even,schw-even-2}, and $T_{ECF}$ is conservative and ergodic on $([0,1],\mu)$, see \cite{flipped}. We choose $k= 1/ \log \sqrt{3}$, so that $E:=\left[0,1/2\right]$ satisfies $\mu(E)=1$. Then we have $E_n= (n/(n+1),(n+1)/(n+2))$. Thus 
\begin{equation}\label{eq:En-ECF}
 \mu(A_{>n}) = \mu(E_n)  \sim \frac{1}{n\log 3}.
\end{equation}
As for $T_{BCF}$, we notice that the map $T_{ECF}$ has one full slope on each of the intervals $I_n:=(1/(n+1), 1/n)$, $n\in \N$, and the induced transformation on $E$ has countably many full slopes on each of the intervals $I_n$ with $n\ge 2$. Therefore the partition $\PPP_{_E}$ induced by $T_{_E}$ is given by the sets $\{A_k \cap I_h\}$ with $k\ge 1$ and $h\ge 2$ and by an analogous argumentation as in Lemma \ref{lemma: psi mixing partition} for $T_{BCF}$ we obtain that $E=[0,1/2]$ induces \mixset. 

The sets $I_n$ are related to the even-integer continued fraction expansion of an irrational number, as one can show that $2h_j =k$ if and only if $T_{ECF}^{j-1}(x) \in I_{2k-1}\cup I_{2k}$, and $\eps_j= -1$ if $x\in I_{2k-1}$ and $\eps_j=+1$ otherwise, see \cite{schw-even}.

Therefore, if we consider the partial sums of the coefficients of the even-integer continued fraction expansion of an irrational $x\in [0,1]$ given by
\[
\sum_{j=1}^N\, 2h_j(x)\qquad \text{ or }\qquad \sum_{j=1}^N\, \left(2h_j(x) + \frac 12(\eps_j(x)-1) \right),
\]
it is the same as considering the Birkhoff partial sums of the observables $g,\tilde{g}:[0,1]\to \N$ defined in \eqref{observ-ecf tilde}-\eqref{observ-ecf} which are constant on each set $I_n$, with $g|_{I_{2n-1}\cup I_{2n}}\equiv 2 n$ and $\tilde{g}|_{I_n}\equiv n$, for the transformation $T_{ECF}$. 

The following calculations are very similar to the ones for the proof of Theorem \ref{cor:BCF} and therefore we only give the main steps. We start with the following lemma which is an analog to Lemma \ref{lem:aux1}:

\begin{lemma} \label{lem:aux2}
Let $E=[0, 1/2]$ and $\varphi_{_E}$ the first return time function to $E$ for $T_{ECF}$. For $M,N\in\mathbb{N}$ we have
\begin{align*}
  \mu\left(\left\{\gamma >N\right\} \cap \left\{\varphi_{_E} >M\right\}  \right)
  \asymp \mu\left(\gamma >N\right) \mu \left(\varphi_{_E} >M \right),
\end{align*}
where $\gamma\in\{g,\tilde{g}\}$ and $\mu$ is the $T_{ECF}$-invariant measure with density $k/(1-x^2)$ for $k= 1/ \log \sqrt{3}$.
\end{lemma}
\begin{proof}
Using the notation as above, since $T_{ECF}$ is piecewise monotone and with full branches with respect to the partition $\{I_m\}$ of $[0,1]$ and $E=\cup_{m\ge 2} I_m$ up to zero-measure sets, one can easily calculate that the level sets of $\varphi_{_E}$ are given by
\begin{align*}
 A_n & = T_{ECF}^{-1}(E_{n-1})\cap E = \bigcup_{m\ge 2} T_{ECF}^{-1}(E_{n-1})\cap I_m =\\
 & =  \left(\bigcup_{r\ge 1} T_{ECF}^{-1}(E_{n-1})\cap I_{2r}\right) \quad \cup \quad \left(\bigcup_{r\ge 2} T_{ECF}^{-1}(E_{n-1})\cap I_{2r-1}\right)\\
 & = \bigcup_{r\ge 1} \left(\frac{n+1}{2r(n+1)+n}\, ,\, \frac{n}{2rn+n-1}\right)
\quad \cup \quad \bigcup_{r\ge 2} \left( \frac{n}{2rn-n+1}\, ,\, \frac{n+1}{2r(n+1)-n}\right), \quad \forall\, n\ge 1.
 \end{align*}
Therefore for all $ M\in \N \cup \{0\}$
 \begin{align*}
 \left\{\varphi_{_E} >M\right\} &= 
 \bigcup_{n\ge M+1} A_n\\
 &= \bigcup_{r\ge 1} \left(\frac{1}{2r+1}\, ,\, \frac{M+1}{2r(M+1)+M}\right)
\quad \cup \quad \bigcup_{r\ge 2} \left( \frac{M+1}{2r(M+1)-M}\, ,\, \frac{1}{2r-1}\right)\\
& = \bigcup_{j\ge 1} \left(\frac{M+1}{2(j+1)(M+1)-M}\, ,\, \frac{M+1}{2j(M+1)+M}\right)
 \end{align*}
 up to zero-measure sets.

Then we recall that the functions $g$ and $\tilde{g}$ are locally constant on the partition $I_m$, therefore for all $N>0$
  \begin{align*}
  \left\{g > N \right\} \cap \left\{\varphi_{_E} >M\right\} = &  \bigcup_{j>N/2} \left(\frac{M+1}{2(j+1)(M+1)-M}\, ,\, \frac{M+1}{2j(M+1)+M}\right) \\  
  &\cup \quad \left( \frac{M+1}{2(\lfloor N/2 \rfloor +1)(M+1)-M}\, ,\, \frac{1}{2\lfloor N/2 \rfloor +1}\right) ,\\
  \left\{\tilde{g} > N \right\} \cap \left\{\varphi_{_E} >M\right\} =
  & \bigcup_{j\ge \lceil N/2 \rceil} \left(\frac{M+1}{2(j+1)(M+1)-M}\, ,\, \frac{M+1}{2j(M+1)+M}\right)\\
  & \cup \quad \begin{cases}
                \left(\frac{M+1}{(N+1)(M+1)-M}\,  ,\, \frac{1}{N-1}\right)         &\text{for $N$ odd,}\\[0.1cm]
                \emptyset&\text{otherwise}.
               \end{cases}
 \end{align*}

Finally, we use that $\mu$ is equivalent to the Lebesgue measure on $E$. Since for $M,N$ big enough
 \begin{align*}
 \lambda\left(g > N \right) = \sum_{m=2\lfloor N/2\rfloor +1}^\infty\, \lambda(I_m) 
\sim \frac{1}{N},\qquad
\lambda\left(\tilde{g} > N \right) = \sum_{m=N+1}^\infty\, \lambda(I_m) \sim \frac{1}{N},
 \end{align*}
 \begin{align*}
   \lambda\left( \varphi_{_E} >M\right) &\, = \sum_{j= 1}^\infty\, \lambda \left(\frac{M+1}{2(j+1)(M+1)-M}\, ,\, \frac{M+1}{2j(M+1)+M}\right)\\
   &\, = \frac{2}{M+1}\, \sum_{j=1}^\infty\, \frac{1}{(2j+\frac{M}{M+1})(2j+2-\frac{M}{M+1})} \asymp \frac 1M,
\end{align*}
and 
 \begin{align*}
  \lambda\left( \left\{g > N \right\} \cap \left\{\varphi_{_E} >M\right\}\right)
  &\asymp \sum_{j>N/2} \lambda \left(\frac{M+1}{2(j+1)(M+1)-M}\, ,\, \frac{M+1}{2j(M+1)+M}\right)\\
  &=  \frac{2}{M+1}\, \sum_{j>N/2}\, \frac{1}{(2j+\frac{M}{M+1})(2j+2-\frac{M}{M+1})} \asymp \frac{1}{MN},\\
  \lambda\left( \left\{\tilde{g} > N \right\} \cap \left\{\varphi_{_E} >M\right\}\right) &\asymp \frac{1}{MN},
 \end{align*}
both statements of the lemma follows.
\end{proof}

\begin{proof}[Proof of Theorem \ref{cor:ECF}]
First, by \eqref{eq:En-ECF} it follows that assumption (i) of Theorem \ref{thm:main1} is fulfilled. Then, the functions $g, \tilde{g}:[0,1]\to \N$ given in \eqref{observ-ecf tilde}-\eqref{observ-ecf} clearly satisfy assumptions (iii) and (iv) of Theorem \ref{thm:main1} with $c=2$. Finally, by \eqref{eq:En-ECF}, assumption (ii) of Theorem \ref{thm:main1} is satisfied with $\kappa= 1$. Thus, if we use that $g(T^{j-1}_{ECF}(x)) =2 h_j(x)$ for all $j\ge 1$ and a.e. $x\in [0,1]$, we obtain 
\[
\lim_{N\to \infty}\, \frac 1N\, \Big( \sum_{j=1}^{N+m(N,E,x)}\, 2h_j(x) -  \max_{1\leq k\leq N+m(N,E,x)} 2h_k(x) - 2\, m(N,E,x) \Big) = 3.
\]
By Lemma \ref{lem:aux2}, also \eqref{eq:cond1} is fulfilled and we obtain the statement of the theorem.

The same arguments apply to $\tilde{g}$ for which $\tilde{g}(T^{j-1}_{ECF}(x)) =2 h_j(x)+1/2(\eps_j-1)$ for all $j\ge 1$ and a.e. $x\in [0,1]$, and the proof is finished.
\end{proof}


\end{document}